\documentclass[10pt, a4paper, reqno]{amsart}
\usepackage{amscd}
\usepackage{dsfont}
\usepackage{algpseudocode,algorithm,algorithmicx}

\algrenewcommand\algorithmicrequire{\textbf{Precondition:}}
\algrenewcommand\algorithmicensure{\textbf{Postcondition:}}

\DeclareMathSymbol{\shortminus}{\mathbin}{AMSa}{"39}

\usepackage{amsmath, upgreek}
\usepackage{amstext}
\usepackage{amsthm}
\usepackage{amssymb}
\usepackage{amsopn}
\usepackage{amssymb}
\usepackage{amsxtra}
\usepackage{amsfonts}
\usepackage{fancyhdr}
\usepackage{paralist, epsfig}
\usepackage[mathcal]{euscript}
\usepackage[english]{babel}
\usepackage[latin1]{inputenc}
\usepackage{bbold}

\usepackage{tikz}
\usepackage{tikz-cd}
\usepackage{tkz-euclide}

\usepackage{euscript}
\makeatletter
\g@addto@macro\normalsize{%
  \setlength\abovedisplayskip{10pt}
  \setlength\belowdisplayskip{10pt}
  \setlength\abovedisplayshortskip{5pt}
  \setlength\belowdisplayshortskip{8pt}
}

\allowdisplaybreaks

\newtheoremstyle{normal}
{5pt}
{5pt}
{\normalfont}
{}
{\bfseries}
{}
{0.4em}
{\bfseries{\thmname{#1}\thmnumber{ #2}.\thmnote{ \hspace{0.5em}(#3)\newline}}}

\newtheoremstyle{kursiv}
{5pt}
{5pt}
{\itshape}
{}
{\bfseries}
{}
{0.4em}
{\bfseries{\thmname{#1}\thmnumber{ #2}.\thmnote{ \hspace{0.5em}(#3)\newline}}}

\theoremstyle{kursiv}

\theoremstyle{normal}

\newtheorem{thm}{Theorem}

\newtheorem{rmk}[thm]{Remark}
\newtheorem{cor}[thm]{Corollary}

\newcommand{\columnvec}[2]{\genfrac{[}{]}{0pt}{}{\,#1\,}{#2}}

\renewcommand{\epsilon}{\varepsilon}

\renewcommand{\theta}{\vartheta}

\newcommand{\Bigsum}[2]{\ensuremath{\mathop{\textstyle\sum}_{#1}^{#2}}}

\newcommand{\e}{\operatorname{e}\nolimits}

\usepackage{setspace}
\usepackage{color}

\definecolor{grey}{gray}{.3}

	\usetikzlibrary{matrix,arrows,decorations,decorations.pathmorphing,positioning,decorations.pathreplacing,shapes}
	\tikzset{commutative diagrams/.cd, 
		mysymbol/.style = {start anchor=center, end anchor = center, draw = none}}

	\newcommand{\commutes}[2][\circ]{\arrow[mysymbol]{#2}[description]{#1}}

\newcommand{\HDLCS}{\operatorname{\mathsf{HD-LCS}}\nolimits}
\newcommand{\LS}{\operatorname{\mathsf{LS}}\nolimits}
\newcommand{\PLS}{\operatorname{\mathsf{PLS}}\nolimits}
\newcommand{\LN}{\operatorname{\mathsf{LN}}\nolimits}
\newcommand{\PLN}{\operatorname{\mathsf{PLN}}\nolimits}
\newcommand{\LSw}{\operatorname{\mathsf{LS}}_{\operatorname{\mathsf{w}}}\nolimits}
\newcommand{\PLSw}{\operatorname{\mathsf{PLS}}_{\operatorname{\mathsf{w}}}\nolimits}

\usepackage{lipsum,eso-pic,xcolor}

\begin{document}
\allowdisplaybreaks
$ $
\vspace{-40pt}

\title{The category of Silva spaces is not integral}

\author{Marianne Lawson\hspace{0.5pt}\MakeLowercase{$^{\text{1}}$} and\hspace{1.5pt} Sven-Ake\ Wegner\hspace{0.5pt}\MakeLowercase{$^{\text{2}}$}}

\renewcommand{\thefootnote}{}
\hspace{-1000pt}\footnote{\hspace{5.5pt}2020 \emph{Mathematics Subject Classification}: Primary 18E05, 46A13; Secondary 46M10, 46A45, 18G80.\vspace{1.6pt}}


\hspace{-1000pt}\footnote{\hspace{5.5pt}\emph{Key words and phrases}: LS-space, PLS-space, integral category, injective and projective objects. \vspace{1.6pt}}

\hspace{-1000pt}\footnote{\hspace{0pt}$^{1}$\,Teesside University, School of Computing, Engineering \&{} Digital Technologies, TS1\;3BX Middlesbrough, UK.\vspace{1.6pt}}

\hspace{-1000pt}\footnote{\hspace{0pt}$^{2}$\,Corresponding author: Universit\"at Hamburg, Fachbereich Mathematik, Bundesstra\ss{}e 55, 20146 Hamburg, Germany,\linebreak\phantom{x}\hspace{1.2pt}phone: +49\,(0)\,40\:42838\:76\:57, e-mail: sven.wegner@uni-hamburg.de.}

\begin{abstract}
We establish that the category of Silva spaces, aka LS-spaces, formed by countable inductive limits of Banach spaces with compact linking maps as objects and linear and continuous maps as morphisms, is not an integral category. The result carries over to the category of PLS-spaces, i.e., countable projective limits of LS-spaces---which contains prominent spaces of analysis such as the space of distributions and the space of real analytic functions. As a consequence, we obtain that both categories neither have enough projective nor enough injective objects. All results hold true when `compact' is replaced by `weakly compact' or `nuclear'. This leads to the categories of PLS-, PLS$_{\text{w}}$- and PLN-spaces, which are examples of `inflation exact categories with admissible cokernels' as recently introduced by Henrard, Kvamme, van Roosmalen and the second-named author.
\end{abstract}


\maketitle

\vspace{-15pt}

\section{Introduction}\label{SEC-1}

Integral categories were introduced by Rump \cite{Rump01} in 2001 and have since caught attention on the one hand in representation theory, see, e.g., Br\"ustle, Hassoun, Tattar \cite{BHT21}, or Nakaoka \cite{Nakaoka}, and on the other hand in functional analysis, see Hassoun, Shah, Wegner \cite{HSW}. By definition, a pre-abelian category $\mathcal{A}$ is \emph{right integral}, if monomorphisms  pushout to monomorphism, meaning whenever $f\colon X\rightarrow Y$ is monic and $g\colon X\rightarrow Z$ is arbitrary, then in the diagram below
\begin{equation*}
\begin{tikzcd}
X\arrow{r}{f}\arrow{d}[swap]{g} \commutes[\text{PO}]{dr}& Y \arrow{d}{j}\\ Z \arrow{r}[swap]{h} & P
\end{tikzcd}
\end{equation*}
the morphism $h$ is again monic. \emph{Left integral} categories are defined dually.

\medskip

The aim of this paper is to establish that the categories $\LS$, $\PLS$, $\LSw$, $\PLSw$, $\LN$ and $\PLN$, whose definitions will be given in due course, are all neither right nor left integral, and to outline consequences of this. The latter categories have proved to be very important in functional analysis as they arise naturally when $\operatorname{Ext}^1$-techniques are used to solve linear partial differential equations, see, in particular, Wengenroth \cite{JochensBuch} and Dierolf, Sieg \cite{DS16}. We will also put the above categories in the context of \emph{inflation exact categories} that were introduced recently by Henrard, Kvamme, van Roosmalen, Wegner \cite{HRKW}. 

\medskip

To keep this article short, we only recall the main implications
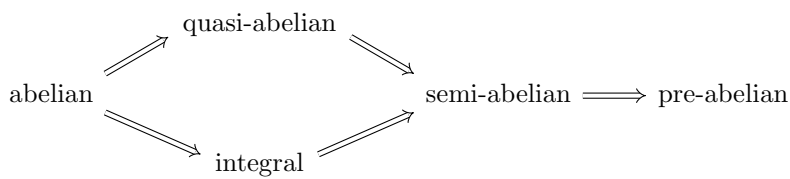
\vspace{-6pt}
\[
\begin{tikzcd}[row sep=0.5cm]
&\hspace{-4pt}\begin{array}{c}\text{quasi-abelian}\\\vspace{-17pt}\end{array}\hspace{-4pt}\arrow[Rightarrow, start anchor=south east, end anchor=north west,pos=0.4]{dr}{}&&\\
\text{abelian} \arrow[Rightarrow, start anchor=north east, end anchor=south west]{ur}\arrow[Rightarrow, start anchor=south east, end anchor=north west]{dr}&& \text{semi-abelian} \arrow[Rightarrow]{r}&\text{pre-abelian}\\
&\hspace{-4pt}\begin{array}{c}\vspace{-18pt}\\\text{integral}\end{array}\hspace{-4pt}\arrow[Rightarrow, start anchor=north east, end anchor=south west]{ur}&&
\end{tikzcd}\vspace{-6pt}
\]
between major types of additive categories. For details we refer to \cite{HSW} where also a long list of examples and non-examples is given. These show in particular that none of the above implications are equivalences and that in general no implication between `quasi-abelian' and `integral' holds true.

\vspace{5pt}

\section{Core result}


A Hausdorff locally convex space $X$ is said to be an \emph{LS-space} (named after Jos\'e Sebasti\~ao e Silva \cite{Silva}), if there exists a spectrum of Banach spaces $X_1\rightarrow X_2\rightarrow X_3\rightarrow\cdots$ linked by linear, injective and compact maps, such that $X=\operatorname{ind}_{n\in\mathbb{N}}X_n$ is the locally convex inductive limit of the spectrum. 
We denote by $\LS$ the category which has the LS-spaces as objects and linear and continuous maps as morphisms. Since the LS-property inherits to closed subspaces, to quotients by closed subspaces, as well as to finite products (Floret \cite[table on p.\ 182]{Floret}), one gets that $\LS$ is a full additive subcategory of the category $\HDLCS$ formed by all Hausdorff locally convex spaces, and that it reflects the kernels and cokernels of $\HDLCS$. As $\HDLCS$ is quasi-abelian, it follows from the latter that $\LS$ is quasi-abelian, too (use, e.g., Frerick, Sieg \cite[Prop.\ 4.20]{FS}).

\medskip

More concretely, in $\LS$ the kernel of a morphism $f\colon X\rightarrow Y$ is the space $f^{-1}(0)\subseteq X$ endowed with the subspace topology and the cokernel is the quotient $Y/\overline{f(X)}$ furnished with the locally convex quotient topology. Consequently, the pullback is a closed subspace of a product and the pushout is a quotient of a product. Moreover, a morphism in $\LS$ is a kernel if and only if it is injective and has a closed range, and it is a cokernel if and only if it is surjective (use the open mapping theorem, e.g., Meise, Vogt \cite[24.30]{MV}). Having this, one can check also directly that $\LS$ is quasi-abelian \cite[p.\ 96]{JochensBuch}.

\medskip

Our main result  is the following.



\smallskip

\begin{thm}\label{MAIN-THM} The category of LS-spaces is neither left nor right integral.
\end{thm}

\begin{proof}Since the category of LS-spaces is quasi-abelian, it is right integral if and only if it is left integral by \cite[Cor.~on p.~173]{Rump01} (see also \cite[Prop.~2.6]{HSW}). We provide a counterexample to show that $\LS$ is not right integral. For a decreasing sequence $V=(v_n)_{n\in\mathbb{N}}$ of strictly positive functions $v_n\colon\mathbb{N}\rightarrow\mathbb{R}$ we denote by
$$
k^1(V):=\mathop{\operatorname{ind}}_{n\in\mathbb{N}}\,\ell^1(v_n)
$$
the K\"othe co-echelon space of order one. The steps in the inductive limit are the Banach spaces
$$
\ell^1(v_n):=\bigl\{x\in\mathbb{K}^{\mathbb{N}}\:\big|\:\|x\|_{v_n}:=\Bigsum{j=1}{\infty}v_n(j)|x(j)|<\infty\bigr\}
$$
and the linking maps are the inclusions. Due to a classic result by K\"othe, see Bierstedt \cite[Proposition 14]{Klaus}, the space $k^1(V)$ is an LS-space if and only if the sequence $V$ satisfies condition
$$
\text{(S)}\;\;\;\forall\:n\;\exists\:m>n\colon \lim_{j\rightarrow\infty}{\textstyle\frac{\text{\small$v_m(j)$}}{\text{\small$v_n(j)$}}}=0.
$$
We now define two concrete sequences $V=(v_n)_{n\in\mathbb{N}}$ and $W=(w_n)_{n\in\mathbb{N}}$ by putting $v_n(j):=j^{-n}$ and $w_n:=\e^{-jn}$. Then both sequences enjoy (S) and since $\e^{-jn}\leqslant j^{-n}$ holds for all $n,j\in\mathbb{N}$ we get 
$$
\forall\:x\in \ell^1(v_n)\colon \|x\|_{w_n} = \Bigsum{j=1}{\infty}\e^{-jn}|x(j)| \leqslant \Bigsum{j=1}{\infty}j^{-n}|x(j)|= \|x\|_{v_n}<\infty.
$$
This means that $\ell^1(v_n)\subseteq\ell^1(w_n)\subseteq k^1(W)$ holds for every $n\in\mathbb{N}$ with continuous inclusion. But then $k^1(V)\subseteq k^1(W)$ must also hold with continuous inclusion by the universal property of the inductive limit, see, e.g., \cite[Prop. 24.7]{MV}. Let us denote the inclusion map by $i\colon k^1(V)\rightarrow k^1(W)$.

\smallskip

Next, we define the map
$$
\uptheta\colon k^1(V) \rightarrow \mathbb{K}, \hspace{10pt} \uptheta(x): = \Bigsum{j=1}{\infty}\e^{-j}\cdot{}\:x(j)
$$
which is well-defined and continuous: Indeed, for $n\in\mathbb{N}$ we choose $C>0$ such that $\e^{-j}\leqslant Cj^{-n}$ holds for all $j\in\mathbb{N}$. Then we get
$$
\forall\:x\in\ell^1(v_n)\colon |\uptheta(x)|\leqslant \Bigsum{j=1}{\infty}\e^{-j}|x(j)|\leqslant C\Bigsum{j=1}{\infty}j^{-n}|x(j)|=C\|x\|_{v_n}<\infty
$$
which means that $\uptheta\colon\ell^1(v_n)\rightarrow\mathbb{K}$ is well-defined and continuous. Now we again use \cite[Prop.~24.7]{MV}.

\smallskip

We continue by forming the pushout of $i$ along $\uptheta$ which leads, in view of what we noted about cokernels in $\LS$ right before Theorem \ref{MAIN-THM}, to the diagram
\begin{equation}\label{PO}
\begin{tikzcd}
k^1(V)\arrow[]{r}{i}\arrow{d}[swap]{\uptheta} \commutes[\text{PO}]{dr}&k^1(W) \arrow{d}{\psi}\\ \mathbb{K} \arrow{r}[swap]{\varphi} & \frac{k^1(W)\oplus{}\mathbb{K}}{\overline{\operatorname{ran}[i\:\shortminus\uptheta]^{\operatorname{T}}}}
\end{tikzcd}
\end{equation}
where $\varphi$ and $\psi$ are the natural maps. We claim now that $\varphi$ is not injective and thus not a monomorphism in $\LS$. This can be seen by considering the sequence $(x^{(k)})_{k\in\mathbb{N}}\subseteq k^1(V)$ defined by $x^{(k)}(j) = (\e^1\hspace{-2pt}/k, \e^2\hspace{-2pt}/k, \hdots, \e^k\hspace{-2pt}/k, 0, \hdots) $. Applying $\uptheta$ to each member of the sequence leads to
$$
\uptheta(x^{(k)}) =\Bigsum{j=1}{\infty} \e^{-j}\cdot\:x^{(k)}(j)= \Bigsum{j=1}{k} \e^{-j}\cdot\: e^j\hspace{-1.5pt}/k = \Bigsum{j=1}{k}  1/k = 1.
$$
Applying $i$ means considering $(x^{(k)})_{k\in\mathbb{N}}\subseteq k^1(W)$, and since the sequence has only finitely many non-zero entries, it is contained in every step space. We consider it in $\ell^1(w_2)$, where we obtain
$$
\|x^{(k)}\|_{w_2} = \Bigsum{j=1}{\infty}\e^{-2j}|x^{(k)}(j)| = \Bigsum{j=1}{k} \e^{-2j}\cdot\:\e^j\hspace{-3pt}/k \leqslant 1/k\cdot\Bigsum{j=1}{\infty}\e^{-j}\xrightarrow{k\longrightarrow \infty} 0
$$
from whence it follows that $(x^{(k)})_{k\in\mathbb{N}}$ converges to zero in $k^1(W)$. Combining the above implies
$$
{\textstyle\columnvec{\phantom{\shortminus}0}{\shortminus1}}\in\overline{\big\{{\textstyle\columnvec{\phantom{\shortminus}i}{\shortminus\uptheta}}(x)\in k^1(W)\oplus\mathbb{K}\:\big|\:x\in k^1(V)\bigr\}}
$$
which means that $\varphi(\shortminus1)={\textstyle\columnvec{\phantom{\shortminus}0}{\shortminus1}}+\overline{\operatorname{ran}[i\:\shortminus\hspace{-1.5pt}\uptheta]^{\operatorname{T}}}=0$ and establishes the claim.
\end{proof}

\smallskip



\begin{rmk}\label{REM} The method used in the proof above is a refinement of that one applied in \cite[Thm.~3.2]{HSW} and was developed by the first-named author in her thesis \cite{MarianneBSc}. We would like to point out the following two observations to the reader:\vspace{3pt}

\begin{compactitem}

\item[(i)] That we were able to construct, as a counterexample, a pushout diagram with the most simple (non-trivial) LS-space $\mathbb{K}$ in the lower left corner, was not a happy coincidence, but virtually possible `without loss of generality': Let us assume that we are given any counterexample, i.e., a diagram 
\begin{equation*}
\begin{tikzcd}
X\arrow{r}{f}\arrow{d}[swap]{g} \commutes[\text{PO}]{dr}& Y \arrow{d}{j}\\ Z \arrow{r}[swap]{h} & P
\end{tikzcd}
\end{equation*}
in which $f$ is monic and $h$ is not. We pick $0\not=\zeta\in Z$ with $h(\zeta)=0$. Next, we apply the Hahn-Banach theorem which gives us a morphism $\zeta^*\colon Z\rightarrow\mathbb{K}$ with $\zeta^*(\zeta)\not=0$. Finally, we form the pushout of $h$ and $\zeta^*$ to obtain the following diagram
\begin{equation*}
\begin{tikzcd}
X\arrow{r}{f}\arrow{d}[swap]{g} & Y \arrow{d}{j}\\ Z\arrow{d}[swap]{\zeta^*} \arrow{r}{h} \commutes[\text{PO}]{dr}& P\arrow{d}{j'}\\
 \mathbb{K} \arrow{r}[swap]{h'} & P'
\end{tikzcd}
\end{equation*}
where the outer rectangle is a `new' counterexample, now with $\mathbb{K}$ in the lower left corner. Notice that the above argument works in any full subcategory of the category of Hausdorff locally convex spaces that contains $\mathbb{K}$ as an object and which has arbitrary pushouts.

\vspace{3pt}

\item[(ii)] There is no counterexample in which the monomorphism $f$ is a topological embedding: Indeed, using that LS-spaces are always complete (see Floret, Wloka \cite[Kor.\ on p.\ 146]{FW}), $X$ being isomorphic to $f(X)$ implies that $f(X)\subseteq Y$ is closed, and thus $f$ is a kernel. Since $\LS$ is quasi-abelian, $h$ is a kernel and therefore in particular monic.

\end{compactitem}

\end{rmk}

\section{Consequences and further results}

A Hausdorff locally convex space $X$ is a \emph{PLS-space} if $X=\operatorname{proj}_{n\in\mathbb{N}}X_n$ is the projective limit of a strongly reduced spectrum $\dots\rightarrow X_3\rightarrow X_2\rightarrow X_1$ of LS-spaces. We denote by $\PLS$ the category with PLS-spaces as objects and linear and continuous maps as morphisms. For a morphism $f\colon X\rightarrow Y$ the kernel is  computed by the same formula as in $\LS$ but the cokernel is the completion of $Y/\overline{f(X)}$, see Sieg \cite[Proof of Prop.\ 3.1.3]{SiegThesis}. Since $\LS\subseteq\PLS$ reflects cokernels we can nevertheless recycle our counterexample from Theorem \ref{MAIN-THM} as well as the dual counterexample (which must exist as $\LS$ is quasi-abelian) and obtain:

\begin{cor}\label{MAIN-THM-2} The category of PLS-spaces is neither left nor right integral.\hfill\qed
\end{cor}

Applying Buan, Marsh \cite[Prop.\ 3.9]{BuanMarsh} (see also \cite[Prop. 5.2]{HSW}) leads to:

\begin{cor}\label{PROJ-INJ} The categories of LS-spaces and PLS-spaces neither have enough projective objects nor do they have enough injective objects.\hfill\qed
\end{cor}

In contrast to $\LS$, the category $\PLS$ is not even left semi-abelian \cite[Prop. 3.1.6]{SiegThesis}. It thus follows from \cite[p.\ 169, Cor.\ 1]{Rump01} (see also \cite[Prop. 2.5]{HSW}) that $\PLS$ does not have enough quasi-projectives (replace `epimorphism' with `cokernel' everywhere in the definition). It is open if $\PLS$ has enough quasi-injectives, or, taking its natural exact structure $\mathbb{E}:=\mathbb{E}_{\operatorname{max}}=\mathbb{E}_{\operatorname{top}}$, see \cite[Prop.\ 3.2.1]{SiegThesis}, into account, enough $\mathbb{E}$-injectives in the notation of \cite{BHT21, BE} (`inflation' instead of `monomorphism').


\smallskip

Replacing `compact' with `weakly compact' or `nuclear' in the definition of LS-spaces leads to the categories of LS$_{\text{w}}$-spaces, and LN-spaces, respectively. Taking strongly reduced countable projective limits of spaces in the two classes gives rise to PLS$_{\text{w}}$-spaces and PLN-spaces. We have the following full inclusions\vspace{-2pt}
\begin{equation}\label{Scheme}
\begin{tikzcd}[column sep=0pt, row sep=0pt]
\PLN&\subseteq&\PLS&\subseteq&\PLSw \\
\text{\rotatebox{90}{$\subseteq$}}&&\text{\rotatebox{90}{$\subseteq$}}&&\text{\rotatebox{90}{$\subseteq$}}\\
\LN&\subseteq&\LS&\subseteq&\LSw \\
\end{tikzcd}
\end{equation}
of pre-abelian categories. Indeed, in $\LN$ and $\PLN$ kernels and cokernels are computed by the same formulas as in $\LS$ and $\PLS$, respectively (for kernels this follows from Doma\'nski, Vogt \cite[Prop.~1.2]{DV2000}, for cokernels in $\LN$ use \cite[Prop.~1.3]{DV2000}, and for cokernels in $\PLN$ proceed as in \cite[Proof of Prop.~3.1.3]{SiegThesis}). Since the K\"othe matrices $V$ and $W$, which we employed in  Theorem \ref{MAIN-THM}, satisfy condition\vspace{-2pt}
$$
\text{(N)}\;\;\;\forall\:n\;\exists\:m>n\colon 
\Bigsum{j=1}{\infty}{\textstyle\frac{\text{\small$v_m(j)$}}{\text{\small$v_n(j)$}}}<\infty,\vspace{-2pt}
$$
it follows from a classic result by Grothendieck, see Bierstedt \cite[Prop.\ 15]{Klaus}, that $k^1(V)$ and $k^1(W)$ are indeed LN-spaces. The diagram \eqref{PO} is thus a pushout in $\LN$ and in $\PLN$. Since $\LN\subseteq\LS$ reflects kernels and cokernels, $\LN$ is quasi-abelian and thus there exists a dual counterexample in $\LN$. Since $\LN\subseteq\PLN$ reflects kernels and cokernels the example applies in $\PLN$, too. We thus get:
 
\begin{cor} The categories of LN-spaces and PLN-spaces are neither left nor right integral. Furthermore, they neither have enough projective objects nor enough injective objects.\hfill\qed
\end{cor}
 
In the categories $\LSw$ and $\PLSw$ cokernels compute again as in $\LS$ and $\PLS$, respectively. Kernels however do not: Given a morphism $f\colon \operatorname{ind}_{n\in\mathbb{N}}X_n\rightarrow \operatorname{ind}_{n\in\mathbb{N}}Y_n$ of LS$_{\text{w}}$-spaces the kernel is
\begin{equation}\label{ker}
\mathop{\operatorname{ind}}_{n\in\mathbb{N}}\,(f^{-1}(0)\cap X_n)\rightarrow\mathop{\operatorname{ind}}_{n\in\mathbb{N}}\,X_n
\end{equation}
where the topology on the left is in general strictly finer than the subspace topology \cite[Rmk.~3.1.1 and Cor.~3.1.3(i)]{DS16}. For morphisms between PLS$_{\text{w}}$-spaces the same has to be done for every step in the projective limit \cite[Prop.~3.1.5]{DS16}. Nevertheless, in both categories it is true that a morphism $f$ is monic if and only if it is injective, and $\LS\subseteq\LSw$ and $\PLS\subseteq\PLSw$ reflect kernels (as, for LS-spaces, the topology defined in \eqref{ker} coincides with the subspace topology). Moreover, $\LSw$ is quasi-abelian \cite[p.~1307 and Cor.\ 3.1.3(i)]{DS16}. By recycling our example from Theorem \ref{MAIN-THM} a third time we arrive at:

\begin{cor} The categories of LS$_\text{w}$-spaces and PLS$_\text{w}$-spaces are neither left nor right integral. Neither of them have enough projective objects nor enough injective objects.\hfill\qed
\end{cor}

We mentioned already that the categories in the bottom row of \eqref{Scheme} are quasi-abelian. For the top row this is not the case and indeed a particular aim of \cite{DS16, SiegThesis} (see also \cite{Septimui, SW}) was to establish natural exact structures on $\PLS$, $\PLSw$ and $\PLN$. We conclude our article by relating the latter approach to the notion of \emph{inflation-exact categories with admissible cokernels} as defined recently in \cite{HRKW}. This new notion allows for the use of homological algebra with respect to the class $\mathbb{C}_{\text{all}}$ of all kernel-cokernel pairs in cases where the latter is not an exact structure in the classical Quillen sense.

\begin{thm}\label{PROP} Endowed with the class $\mathbb{C}_{\text{all}}$ of all kernel-cokernel pairs as conflations, the categories $\PLS$, $\PLSw$ and $\PLN$ are inflation-exact categories with admissible cokernels. The class $\mathbb{C}_{\text{all}}$ is not an exact structure.
\end{thm}
\begin{proof} All three categories are right quasi-abelian but not left quasi-abelian in the notation of \cite[Dfn.~2.3]{HSW}. For $\PLSw$ this follows from \cite[Prop.~3.1.9]{DS16}. For $\PLS$ and $\PLN$ the `right' part can be proved in exactly the same way. The `not left' part for $\PLS$ follows from the fact that $\PLS$ is not left semi-abelian \cite[Prop.~3.1.6]{SiegThesis}. As the example uses a quotient of the space $\mathcal{D}'(\Omega)$, it applies to $\PLN$, too. Using the dual of  \cite[Rmk.~4.10(1)]{HRKW} establishes the claim.
\end{proof}

We refer to \cite{HRKW} for various implications of Theorem \ref{PROP}, including the means for defining the derived category. The following illustrates the trade-off that comes with switching from a (natural) exact structure $\mathbb{E}$ to the conflation structure $\mathbb{C}_{\text{all}}$ on the categories covered by Theorem \ref{PROP}:

\smallskip

\begin{cor}\label{COR-8} Let $\mathcal{A}\in\{\PLS,\,\PLSw,\,\PLN\}$.
\begin{compactitem}

\vspace{4pt}

\item[(i)] There is a natural triangle equivalence $\mathbf{D}^{\text{b}}(\mathcal{A},\mathbb{C}_{\text{all}})\rightarrow\mathbf{D}^{\text{b}}(\mathcal{RH}(\mathcal{A}))$ with $\mathcal{RH}(\mathcal{A})$  being the heart of a natural t-structure on the derived category.

\vspace{4pt}

\item[(ii)] The natural triangle functor $\mathbf{D}^{\text{b}}(\mathcal{A},\mathbb{E})\rightarrow\mathbf{D}^{\text{b}}(\mathcal{A},\mathbb{C}_{\text{all}})$ is neither an equivalence for the maximal exact structure $\mathbb{E}=\mathbb{E}_{\text{max}}$ nor for any other exact structure.
\end{compactitem}
\end{cor}
\begin{proof} The statements can be proved by copying and dualizing \cite[Thm.~9.3]{HRKW} and \cite[Thm.~9.6]{HRKW}.
\end{proof}

\begin{rmk} From the dual of \cite[Thm.~8.8]{HRKW} it follows that the heart $\mathcal{RH}(\mathcal{A})$ in Corollary \ref{COR-8} may be replaced by the localization $(\operatorname{hMon}\mathcal{A})[\{\text{pulations}\}^{-1}]$. The latter might come in handy, as it can be defined without using t-structures. We refer to \cite{Wegner17} for more details and note that it can be checked as in \cite[Section 3]{Wegner17} that $\PLS$, $\PLSw$ and $\PLN$ are what we there called Waelbroeck categories.

\end{rmk}

\smallskip

\small 

{\sc Acknowledgements.} The authors would like to thank the anonymous referee for her/his careful work and for suggesting to add the comments in Remark \ref{REM}(i)--(ii).


\normalsize

\bibliographystyle{amsplain}
\bibliography{Ref}

\end{document}